\documentclass[a4paper, 10pt]{amsart}
\usepackage{amsmath}
\usepackage{amssymb, color}

\theoremstyle{plain}
\newtheorem{theorem}{\bf Theorem}[section]
\newtheorem{proposition}[theorem]{\bf Proposition}
\newtheorem{lemma}[theorem]{\bf Lemma}
\newtheorem{corollary}[theorem]{\bf Corollary}

\theoremstyle{definition}
\newtheorem{example}[theorem]{\bf Example}

\newtheorem{definition}[theorem]{\bf Definition}

\newtheorem{remarks}[theorem]{\bf Remarks}

\newcommand{\N}{\mathbb N}
\newcommand{\Z}{\mathbb Z}

\DeclareMathOperator{\spec}{spec}

\newcommand{\DP}{\negthinspace : \negthinspace}

\newcommand{\red}{\text{\rm red}}

\renewcommand{\t}{\, | \,}

\newcommand{\RingReg}[1]{#1^{\bullet}}
\newcommand{\Fv}[1]{\mathcal{F}_v(#1)}

\numberwithin{equation}{section}

\begin{document}

\title{On $v$-Marot Mori rings and C-rings}

\address{Institut f\"ur Mathematik und Wissenschaftliches Rechnen\\
Karl--Fran\-zens--Universit\"at Graz, NAWI Graz \\
Heinrichstra{\ss}e 36\\
8010 Graz, Austria}

\email{alfred.geroldinger@uni-graz.at,  sebastian@ramacher.at, andreas.reinhart@uni-graz.at}

\author{Alfred Geroldinger and Sebastian Ramacher and Andreas Reinhart}

\thanks{This work was supported by
the Austrian Science Fund FWF P26036-N26.}



\begin{abstract}
C-domains are defined via class semigroups, and every C-domain is a Mori domain with nonzero conductor whose complete integral closure is a Krull domain with finite class group. In order to extend the concept of C-domains to rings with zero divisors, we study $v$-Marot rings as generalizations of ordinary Marot rings and investigate their theory of regular divisorial ideals. Based on this we establish a generalization of a result well-known for integral domains. Let $R$ be a $v$-Marot Mori ring, $\widehat R$ its complete integral closure, and suppose that the conductor $\mathfrak f = (R : \widehat R)$ is regular. If the residue class ring $R/\mathfrak f$ and the class group $\mathcal C (\widehat R)$ are both finite, then $R$ is a C-ring. Moreover, we study both $v$-Marot rings and C-rings under various ring extensions.
\end{abstract}

\maketitle


\bigskip
\section{Introduction} \label{1}
\bigskip

Arithmetical studies of noetherian or, more generally, of Mori domains split into two cases. First suppose that the domain is completely integrally closed. Then it is a Krull domain, the monoid of $v$-invertible $v$-ideals is free abelian, and there is a transfer homomorphism from the domain to the monoid of zero-sum sequences over a subset $G_P$ of the class group. This and a finiteness assumption on the subset $G_P$ are the basis for a variety of arithmetical finiteness results for Krull domains.

First arithmetical investigations of not completely integrally closed Mori domains were restricted to  one-dimensional domains and then to  weakly Krull domains. The concept of C-monoids opened the door to arithmetical investigations of higher dimensional Mori domains.
A C-monoid is a submonoid of a factorial monoid with finite class semigroup, and a domain is a C-domain if its multiplicative monoid is a C-monoid.
If $R$ is a C-domain, then $R$ is a Mori domain, its complete integral closure  $\widehat R$ is a Krull domain with finite class group and the conductor $\mathfrak f = (R : \widehat R)$ is nonzero (see \cite{Ge-HK06a, Re13a}).
The finiteness of the class semigroup allows  to derive similar arithmetical finiteness results for  C-monoids  as they are known for Krull monoids with finite class group (see Proposition \ref{4.2} for a summary, or \cite{Ge-HK06a, Fo-Ge05, Fo-Ha06a, Ga-Ge09b}).

Krull rings with zero divisors were introduced independently by Kennedy \cite{Ke80} and by Portelli and Spangher \cite{Po-Sp83a}, and they found a solid treatment in Huckaba's monograph \cite{Hu88a} in the setting of Marot rings.  The theory of Krull and Mori rings  was further developed (even without requiring the Marot property) by Chang, Glaz, Kang, Lucas, and others (\cite{Ka91a, Ka93a,  Os99a,  Ka00a, Ch01a, Ch-Ka02a, Gl03a, Lu04a, Lu05a, Lu07a, Ch10a}).

Factorization theory in rings with zero divisors was initiated by Daniel D. Anderson in the 1980s. This research was continued by various  authors (e.g., \cite{An-Ma85a,An-Ma85b, An-VL96a, An-VL97a, An-Ag-VL01, An-Ch11a, Ch-An-VL11a, Fr-Fr11a, Ch-Sm13a, Mo14a}), but in comparison to the domain case our knowledge on the arithmetic of rings with zero divisors is still very rudimentary. One possible approach is to focus on the monoid of regular elements, which definitely makes sense if the ring has   few zero divisors. If $R$ is a Marot Krull ring, then its monoid of regular elements is a Krull monoid (\cite{HK93f}, Corollary \ref{3.6}). Thus all arithmetical results for Krull monoids hold true for the monoid of regular elements of a Marot Krull ring. So far only little is known on the arithmetic of regular elements in the non Krull case.

The present paper provides a systematic approach towards the arithmetic of regular elements in a Mori ring which is not Krull.
Let $R$ be a Mori ring and $R^{\bullet}$ the monoid of its regular elements.
The relationship between regular divisorial ideals of $R$ and  regular divisorial ideals of the monoid $R^{\bullet}$ is crucial for our strategy. For this reason we study $v$-Marot rings (Definition \ref{3.2}) which turn out to be precisely those commutative rings for which there is a canonical semigroup isomorphism between the semigroup of regular divisorial ring ideals and the semigroup of regular divisorial ideals of $R^{\bullet}$ (Theorem \ref{3.5}). In Section \ref{3} we study the theory of regular divisorial ideals of $v$-Marot rings.

In Section \ref{4} we introduce C-rings as commutative rings whose monoid of regular elements is a C-monoid. A $v$-Marot C-ring $R$ turns out to be a Mori ring whose complete integral closure $\widehat R$ is Krull with finite class group and with regular conductor $\mathfrak f = (R : \widehat R)$ (Corollary \ref{4.4}). Our main result offers a partial converse which was well-known in the domain case (\cite[Theorem 2.11.9]{Ge-HK06a}). Indeed, if $R$ is a $v$-Marot Mori ring with the above properties and if in addition the residue class ring $R/\mathfrak f$ is finite, then $R$ is a C-ring (Theorem \ref{4.8}). Thus the monoid of regular elements of such a $v$-Marot Mori ring satisfies all arithmetical finiteness results of C-monoids (a summary of such results is given in Proposition \ref{4.2}).

\bigskip
\section{Preliminaries on Marot, Mori, and Krull rings} \label{2}
\bigskip

We denote by $\N$ the set of positive integers and set $\N_0 = \N \cup \{0\}$. For integers $a, b \in \Z$, we denote by $[a, b ] = \{x \in \Z \mid a \le x \le b\}$ the discrete interval.

\noindent
{\bf Semigroups.} By a {\it semigroup}, we mean a multiplicatively written, commutative semigroup with unit element. By a {\it monoid}, we mean a cancellative semigroup. Our notation and terminology are consistent with \cite{Ge-HK06a}. We briefly gather some key notions. Let $S$ be a monoid. Then $S^{\times}$ denotes the unit group, and for any subgroup $T \subset S^{\times}$,  $S/T = \{ s T \mid s \in T \}$ forms a commutative semigroup (with an identity) in a natural sense. If $T = S^{\times}$, then
$S _{\red} = \{ s S^{\times} \mid s \in S \}$ is  the associated reduced semigroup. We say that $S$ is reduced if $S^{\times} = \{1\}$.

Let $H$ be a monoid. Then $\mathsf q (H)$ denotes the quotient group of $H$, and $\widehat H \subset \mathsf q (H)$ the complete integral closure of $H$. For subsets $X,\, Y \subset \mathsf q(H)=Q$ we set
\[
(Y :_Q X) = (Y \DP X) = \big\{ a\in Q \, \big| \, a X \subset Y\big\}\,, \;\; X^{-1}
= (H \DP X)\, , \;\; \text{and} \;\; X_{v_H} = X_v = (X^{-1})^{-1}\,.
\]
We say that $X$ is {\it $H$-fractional} if there is some $c \in H$ such that $cX \subset H$,  that $X$ is a {\it fractional $v$-ideal} of $H$ if $X$ is $H$-fractional and $X_v = X$, and that $X$ is a $v$-ideal of $H$ if $X \subset H$ and $X_v = X$.
We denote by  $v$-$\spec (H)$ the set of prime $v$-ideals of $H$, by $(\mathcal F_v (H), \cdot_v)$ the semigroup of fractional $v$-ideals of $H$ with $v$-multiplication, and by  $\mathcal I_v (H)$ the subsemigroup of $v$-ideals of $H$.
  Then $\mathcal I_v^* (H) = \mathcal I_v (H) \cap \mathcal F_v (H)^{\times}$ is the monoid of $v$-invertible $v$-ideals of $H$ and its quotient group is $\mathcal F_v (H)^{\times}$. In any context, the terms $v$-ideal and divisorial ideal will be used synonymously.

The monoid $H$ is called a
\begin{itemize}
\item {\it Mori monoid} ($v$-{\it noetherian} resp.) if it satisfies the ACC
(ascending chain condition) on $v$-ideals,

\item {\it Krull monoid} if it is a completely integrally closed Mori monoid.
\end{itemize}

\smallskip
\noindent
{\bf Class groups.} Let $F$ be a monoid and $H \subset F$ a submonoid.
We say that $H \subset F$ is
\begin{itemize}
\item {\it saturated} if $H = F \cap \mathsf q (H)$,

\item {\it cofinal} if for all $a \in F$ there is an element $b \in H$ such that $a \t b$.
\end{itemize}
For every $a \in F$ we set $[a]_{F/H} = a \mathsf q (H) \in \mathsf q (F)/\mathsf q (H)$, and we define
\[
F/H = \{ [a]_{F/H} \mid a \in F \} \subset \mathsf q (F)/\mathsf q (H) \,.
\]
Then $H \subset F$ is cofinal if and only if $F/H$ is a group. In particular, if $F/H$ is finite or if $\mathsf q (F)/\mathsf q (H)$ is a torsion group, then $F/H = \mathsf q (F)/\mathsf q (H)$.

Now suppose that $H$ is a Mori monoid and set $\mathcal H = \{aH \mid a \in H\}$. Then $\mathcal H \subset \mathcal I_v^* (H)$ is saturated and cofinal, and all above concepts coincide with the usual $v$-class group of $H$. Indeed, we have
\[
\mathcal C_v (H) = \mathcal F_v (H)^{\times}/\mathsf q (\mathcal H) = \mathcal I_v^* (H)/\mathcal H \,.
\]

\smallskip
\noindent
{\bf Rings.} By a {\it ring}, we mean a commutative ring with unit element. Let $R$ be a ring. We denote by $R^{\times}$ the group of invertible elements of $R$, by $\mathsf T (R)$ the total quotient ring of $R$, by $\mathsf Z (R)$ the set of zero divisors, by $\overline R$ the integral closure of $R$ in $\mathsf T (R)$, and by $\widehat R$ the complete integral closure of $R$ in $\mathsf T (R)$. For a subset $X \subset \mathsf T (R)$, we denote by $X^{\bullet} = X \setminus \mathsf Z ( \mathsf T (R) )$  the set of all regular elements of $X$, and we say that $X$ is regular if $X^{\bullet} \ne \emptyset$.
Clearly, the set of regular elements $R^{\bullet}$ of $R$ is a monoid, and $\mathsf T (R)^{\bullet} = \mathsf T (R)^{\times} = \mathsf q (R^{\bullet})$.

Let $X, Y \subset \mathsf T (R) = T$ be subsets.  We define
\[
X Y = \{xy \mid x\in X, y\in Y \} \quad \text{and} \quad X^k = \{\prod_{i=1}^k x_i\mid x_1, \ldots, x_k \in X \} \ \text{ for all} \ k\in\N \,.
\]
Moreover, $X$ is called {\it $R$-fractional} if $c X \subset R$ for some $c\in R^{\bullet}$. If $X \subset T^{\bullet}$, then $X$ is regular if and only if $X\not=\emptyset$, and   $X$ is $R$-fractional if and only if $X$ is $R^{\bullet}$-fractional. We set
\[
(Y :_T X) = (Y \DP X) = \big\{ a\in  T  \, \big| \, a X \subset Y\big\}\,, \;\; X^{-1}
= (R \DP X)\, , \;\; \text{and} \;\; X_{v_R} = X_v = (X^{-1})^{-1}\,.
\]
We say that $X$ is a regular fractional $v$-ideal  of $R$ if $X$ is regular, $R$-fractional, and $X_v = X$, and that $X$ is a regular $v$-ideal  of $R$ if $X \subset R$ is regular and $X_v = X$. We denote by $v$-$\spec (R)$ the set of regular prime $v$-ideals of $R$, by $(\mathcal F_v (R), \cdot_v)$ the semigroup of regular fractional $v$-ideals of $R$ with $v$-multiplication, and by $\mathcal I_v (R)$ the subsemigroup of regular $v$-ideals of $R$. Then $\mathcal I_v^* (R) = \mathcal I_v (R) \cap \mathcal F_v (R)^{\times}$ is the monoid of $v$-invertible regular $v$-ideals of $R$ and its quotient group equals $\mathcal F_v (R)^{\times}$.   Note that if
 $X$ is regular and $R$-fractional, then $X_v \in\mathcal{F}_v(R)$,  and every regular fractional $v$-ideal is a regular submodule of $\mathsf T (R)$. In any context, the terms $v$-ideal and divisorial ideal will be used synonymously.

The ring $R$ is called a
\begin{itemize}
\item {\it Marot ring} if every regular ideal of $R$ is generated by its regular elements,

\item {\it Mori ring} if it satisfies the ACC on regular divisorial ideals of $R$,

\item {\it Krull ring} if it is a completely integrally closed Mori ring.
\end{itemize}
Finite direct products of domains, noetherian rings, polynomial rings over arbitrary commutative rings are Marot rings, and all overrings of Marot rings are Marot rings. This and various characterizations of Marot rings can be found in Huckaba's book \cite{Hu88a}.
For more information on Mori rings we refer to the work of Lucas \cite{Lu04a, Lu05a, Lu07a}, and in particular to the characterization given in   \cite[Theorem 2.22]{Lu04a}. For a characterization of Krull rings we refer to the work of Kang \cite[Theorem 13]{Ka00a} (note that older concepts of Krull rings -- as given in \cite{Po-Sp83a} and in Huckaba's monograph \cite{Hu88a} -- coincide with the present one in the setting of Marot rings).

\smallskip
It is well-known that the $v$-system on a commutative ring still has many of the nice properties of a $v$-system on a domain. Formally speaking, the $v$-system on a commutative ring $R$ is a weak ideal system
(in the sense of \cite{HK98}). This can be seen from the next lemma. However, since we will use only the $v$-system in this paper, we do not give the abstract definition of (weak) ideal systems. For their properties and their relationship to star and semistar operations, the interested reader may want to consult the survey article \cite{HK11b}.

Since the relationship between the regular divisorial ideals of a ring $R$ and the regular divisorial ideals of its regular monoid $R^{\bullet}$ is crucial in the present paper, we provide a careful analysis of these properties with full proofs. We start with two lemmas (parts of which are definitely well-known) and a remark, where we remind of two main differences between the $v$-system on a general ring and the $v$-system on a domain.

\medskip
\begin{lemma} \label{2.1}
Let $R$ be a ring, $T=\mathsf{T}(R)$, $X,Y\subset T$, and $c\in T$.
\begin{enumerate}
\item $X \cup \{0\} \subset X_v$ and $(X_v)^{-1}=X^{-1}$.

\item $X \subset Y_v$ implies that $X_v \subset Y_v$.

\smallskip
\item $(XY)^{-1}=(X^{-1}:_T Y)$.

\smallskip
\item $(XY)_v=(X_vY)_v$.

\smallskip
\item $cX_v\subset (cX)_v$ and if $c\in T^{\bullet}$, then $cX_v=(cX)_v$.

\smallskip
\item $(X_v:_T Y)=(X_v:_T Y_v)$.
\end{enumerate}
\end{lemma}

\begin{proof}
1. Let $x\in X$ and $y\in X^{-1}$. Then $xy\in R$, hence $xX^{-1}\subset R$, and thus $x\in X_v$. Next we show that $(X_v)^{-1}=X^{-1}$. Let $x\in (X_v)^{-1}$ be given. Then $xy\in R$ for all $y\in X_v$. Therefore, $xy\in R$ for all $y\in X$, hence $x\in X^{-1}$. Conversely, let $x\in X^{-1}$. Then $xy \in R$ for all  $y\in X_v$.  Consequently, $xX_v\subset R$, and thus $x\in (X_v)^{-1}$.

\smallskip
2.  It follows by 1. that $Y^{-1}=(Y_v)^{-1}\subset X^{-1}$. Therefore, $X_v = (X^{-1})^{-1} \subset (Y^{-1})^{-1} = Y_v$.

\smallskip
3. Let $x\in T$. Then $x\in (XY)^{-1}$ if and only if $xXY\subset R$ if and only if $xY\subset X^{-1}$ if and only if $x\in (X^{-1}:_T Y)$.

\smallskip
4. It follows by 1. and 3. that $(X_vY)_v=((X_vY)^{-1})^{-1}=((X_v)^{-1}:_T Y)^{-1}=(X^{-1}:_T Y)^{-1}=(XY)_v$.

\smallskip
5. By 1. and 4. we have $cX_v\subset (cX_v)_v=(\{c\}X_v)_v=(\{c\}X)_v=(cX)_v$. Now let $c\in T^{\bullet}$ and $x\in T$. Then $x\in (cX)^{-1}$ if and only if $xcX\subset R$ if and only if $cx\in X^{-1}$ if and only if $x\in c^{-1}X^{-1}$. Therefore, $(cX)^{-1}=c^{-1}X^{-1}$. Consequently, $(cX)_v=((cX)^{-1})^{-1}=(c^{-1}X^{-1})^{-1}=(c^{-1})^{-1}(X^{-1})^{-1}=cX_v$.

\smallskip
6. Note that $(X_v:_T Y_v)=((R:_T X^{-1}):_T Y_v)=(R:_T X^{-1}Y_v)=((R:_T Y_v):_T X^{-1})=((R:_T Y):_T X^{-1})=(R:_T X^{-1}Y)=(X_v:_T Y)$ by 1. and 3.
\end{proof}

\smallskip
In general, the $v$-system is not an ideal system. This means that in general we do not have $cX_v=(cX)_v$ for all $c \in T$,  as we point out in the following remark.

\medskip
\begin{remarks}~ \label{2.2}

1. We provide an example of a ring $R$ for which  there are $c \in R$ and $X \subset R$ such that $cX_v \subsetneq (cX)_v$.  Let $R$ be a ring with $R = \mathsf Z (R) \cup R^{\times}$, and let $a, b \in \mathsf Z (R) \setminus \{0\}$ with $a b = 0$. Then $\mathsf T (R) = R$,
\[
\{a\}_v = (R :_T (R:_T \{a\})) = (R :_T R) = R \,,
\]
$b\{a\}_v = bR$, but $\{ab\}_v = \{0\}_v = R$ whence $b\{a\}_v \subsetneq \{ab\}_v$.

\smallskip
2. In general, a divisorial ideal need not be the intersection of the fractional principal ideals containing  it. The first example is due to Daniel D. Anderson (\cite[Example]{An-Ma85b}; see also \cite[Section 27, Example 11]{Hu88a}). However, $v$-Marot rings -- as considered in Section \ref{3} -- do have this property.
\end{remarks}

\medskip
\begin{lemma} \label{2.3}
Let $R$ be a ring,  $T=\mathsf{T}(R)$, $H=R^{\bullet}$, $Q=\mathsf{q}(H)$, and $\mathfrak{f}=(R:_T\widehat{R})$.
\begin{enumerate}
\item For every $R$-module $M \subset T$ and every subset $Y \subset T$, we have $(M \DP_T \ {_R\langle} Y^{\bullet}\rangle)^{\bullet} = (M^{\bullet} \DP_Q \ Y^{\bullet})$ and $(M \DP_T \ {_R\langle} Y^{\bullet}\rangle) = (M \DP_T \ Y^{\bullet})$.

\smallskip
\item For every  finite subset  $E\subset\widehat{R}$, we have $E_{v_R}\subset\widehat{R}$,  and $(\widehat{R})_{v_R}=(\mathfrak{f}:_T\mathfrak{f})$.

\smallskip
\item For every subset $\mathfrak{a}\subset Q$, we have $(R\DP_T\,\mathfrak{a})^{\bullet} = (H\DP_Q\,\mathfrak{a})$ and $(\mathfrak{a}_{v_R})^{\bullet}\subset\mathfrak{a}_{v_H}$.
\end{enumerate}
\end{lemma}

\begin{proof}
1. By definition, we have $(M \DP_T \ {_R\langle} Y^{\bullet}\rangle)^{\bullet} \subset (M^{\bullet} \DP_Q \ Y^{\bullet})$. Conversely, let $z \in (M^{\bullet} \DP_Q \ Y^{\bullet})$. If $x \in {_R\langle} Y^{\bullet} \rangle$, then $x = \lambda_1 y_1 + \ldots, + \lambda_m y_m$, where $m \in \N$, $\lambda_1, \ldots, \lambda_m \in R$, and $y_1, \ldots, y_m \in Y^{\bullet}$. Then $zy_1, \ldots, zy_m \in M^{\bullet}$, and $zx = \lambda_1 z y_1 + \ldots, + \lambda_m z y_m \in M$. The proof of the second equation runs along the same lines.

2. Let $E\subset\widehat{R}$ be finite. There is some sequence $(c_e)_{e\in E}$ of regular elements of $R$ such that $c_ee^k\in R$ for all $e\in E$ and $k\in\mathbb{N}$. Set $c=\prod_{e\in E} c_e$. It follows that $c\in R^{\bullet}$ and $cE^k\subset R$ for all $k\in\mathbb{N}$. Let $x\in E_v$. Then $cx^k\in c(E_v)^k\subset (c(E_v)^k)_v=(cE^k)_v\subset R$ for all $k\in\mathbb{N}$, hence $x\in\widehat{R}$. We have $(\mathfrak{f}:_T\mathfrak{f})=((R:_T\widehat{R}):_T\mathfrak{f})=(R:_T\widehat{R}\mathfrak{f})=(R:_T\mathfrak{f})=(R:_T (R:_T\widehat{R}))=(\widehat{R})_{v_R}$.

\smallskip
3. If $a\in (R\DP_T\,\mathfrak{a})^{\bullet}$, then $a\in T^{\bullet}$ with $a\mathfrak{a}\subset R$, hence $a\mathfrak{a}\subset R^{\bullet}$, and thus $a\in (R^{\bullet}\DP_Q\,\mathfrak{a})$. It is obvious that $(H:_Q\mathfrak{a})\subset (R:_T\mathfrak{a})^{\bullet}$. Let $x\in (\mathfrak{a}_{v_R})^{\bullet}$. Then $x(H:_Q\mathfrak{a})\subset x(R:_T\mathfrak{a})\subset R$, hence $x(H:_Q\mathfrak{a})\subset R\cap Q=H$. Therefore, $x\in\mathfrak{a}_{v_H}$.
\end{proof}

\bigskip
\section{On $v$-Marot rings} \label{3}
\bigskip

A finitary weak module system $r$ on a commutative ring $R$ gives rise to a weak module system on the monoid $R^{\bullet}$, and $R$ is said to be an $r$-Marot ring if every regular $r$-module is generated by its regular elements (if the $d$-system denotes the system of classical ring ideals, then the notion of a $d$-Marot ring coincides with the notion of an ordinary Marot ring). These abstract concepts were introduced in \cite{HK03a} to study Dedekind and Pr\"ufer monoids without cancellation.

In this section we study $v$-Marot rings which were first considered in \cite{An-Ma85b} (they are precisely the rings satisfying Property (D)).  Clearly, $v$-Marot rings generalize ordinary Marot rings and allow an inclusion preserving isomorphism between the semigroup of regular divisorial ring ideals and the semigroup of regular divisorial ideals of the associated monoid (see Theorem \ref{3.5}). Therefore $v$-Marot rings provide the perfect setting for the study of C-rings done in Section \ref{4}.

\medskip
\begin{lemma} \label{3.1}
Let $R$ be a ring and $T = \mathsf T (R)$. Then the following statements are equivalent{\rm \,:}
\begin{enumerate}
\item[(a)] For every $I \in \mathcal I_v (R)$, we have $I=(I^{\bullet})_v$ $($i.e., each regular divisorial ideal of $R$ is $v$-generated by its regular elements$)$.

\smallskip
\item[(b)] For every $I\in\mathcal{F}_v(R)$, we have $I=(I^{\bullet})_v$.

\smallskip
\item[(c)]For every $I\in\mathcal{F}_v(R)$, we have $I=\bigcap_{z\in T^{\bullet},zR\supset I} zR$.
\end{enumerate}
\end{lemma}

\begin{proof}
 (a) $\Rightarrow$ (b) Let $I\in\mathcal{F}_v(R)$. Then $cI\in\mathcal{I}_v(R)$ for some $c\in R^{\bullet}$. It follows that $c(I^{\bullet})_v=(cI^{\bullet})_v=((cI)^{\bullet})_v=cI$, and thus $(I^{\bullet})_v=I$.

\smallskip
 (b) $\Rightarrow$ (c) Let $I\in\mathcal{F}_v(R)$. Then $I^{-1}\in\mathcal{F}_v(R)$. Set $E=(I^{-1})^{\bullet}$. Then $I^{-1}=E_v$, hence $I=E^{-1}=\bigcap_{f\in E} f^{-1}R$. This implies that $I\subset\bigcap_{z\in T^{\bullet},zR\supset I} zR\subset\bigcap_{f\in E} f^{-1}R=I$. Therefore, $I=\bigcap_{z\in T^{\bullet},zR\supset I} zR$.

\smallskip
 (c) $\Rightarrow$ (a) Let $I\in\mathcal{I}_v(R)$. Then $I^{-1}\in\mathcal{F}_v(R)$. Set $F=\{z\in T^{\bullet}\mid zR\supset I^{-1}\}$, and $E=\{z^{-1}\mid z\in F\}$. It follows that $I^{-1}=\bigcap_{z\in F} zR=\bigcap_{z\in E} z^{-1}R=E^{-1}$, hence $I=E_v$. Therefore, $I=E_v\subset (I^{\bullet})_v\subset I$, and thus $I=(I^{\bullet})_v$.
\end{proof}

\medskip
\begin{definition} \label{3.2}
A ring $R$ is called a {\it $v$-Marot ring} if it satisfies the equivalent conditions of Lemma \ref{3.1}.
\end{definition}

\smallskip
The next proposition first points out that every Marot ring is a $v$-Marot ring. Moreover,
in our main results  we will often assume that the conductor $(R\DP\widehat{R})$ of the given ring $R$ is  regular, and every $v$-Marot C-ring will have this property (see Corollary \ref{4.4}). Based on the work of Chang and Kang, the next proposition demonstrates that this assumption holds true in very general but  natural settings.

\medskip
\begin{proposition} \label{3.3}
Let $R$ be a ring.
\begin{enumerate}
\item If $R$ is noetherian, then $R$ is a Marot ring, and if $R$ is a Marot ring, then $R$ is a $v$-Marot ring.

\smallskip
\item Let $S$ be a ring with $R\subset S\subset\mathsf{T}(R)$ and suppose that $S$ is a finitely generated $R$-module. Then $(R : S)$ is regular and we have
      \begin{enumerate}
      \item $R$ is noetherian if and only if $S$ is noetherian.

      \item Every regular ideal of $R$ is finitely generated if and only if every regular ideal of $S$ is finitely generated.
      \end{enumerate}

\smallskip
\item Suppose that every regular ideal of $R$ is finitely generated. Then $R$ is a Mori ring, its integral closure $\overline{R}$ is a Krull ring, and hence $\overline{R}=\widehat{R}$. Moreover,  $\mathfrak{f}=(R \DP \widehat{R})$ is regular if and only if $\widehat{R}$ is a finitely generated $R$-module.
\end{enumerate}
\end{proposition}

\begin{proof}
1. A proof of the first statement can be found in \cite[Theorem 7.2]{Hu88a}. Let $R$ be a Marot ring and $I\in\mathcal{I}_v(R)$. By Lemma \ref{2.3}.1, we infer that $I=I_v=({_R}\langle I^{\bullet}\rangle)_v=(I^{\bullet})_v$.

2. Statement (a) is the Theorem of Eakin-Nagata, and Statement (b) is a generalization due to Chang \cite{Ch10a}. If $S = {_R\langle} a_1b_1^{-1}, \ldots, a_sb_s^{-1} \rangle$ where $s \in \N$, $a_1, \ldots, a_s \in R$, and $b_1, \ldots, b_s \in R^{\bullet}$, then $b = b_1 \cdot \ldots \cdot b_s \in R^{\bullet}$, $b S \subset R$, and hence $b \in (R \DP S)$.

3. If $(I_n)_{n \ge 0}$ is an ascending chain of regular divisorial ideals, then the union $\cup I_n$ is a regular ideal, hence finitely generated, and thus the chain becomes stationary. The integral closure of $R$ is a Krull ring by \cite{Ch-Ka02a}. Thus $\overline R$ is completely integrally closed and hence $\overline R = \widehat R$. If $\overline R$ is a finitely generated $R$-module, then $(R : \overline R)$ is regular by 2. Suppose that $\mathfrak f$ is a regular ideal, and let $f \in \mathfrak f^{\bullet}$. Then $f \widehat R \subset R$ is a regular ideal, hence a finitely generated $R$-module by assumption and the same is true for the isomorphic $R$-module $\widehat R$.
\end{proof}

\smallskip
A ring, whose regular ideals are finitely generated, is always Mori but it need not be a $v$-Marot ring (see \cite[Example 11, Section 27]{Hu88a}).

\medskip
\begin{lemma} \label{3.4}
Let $R$ be a $v$-Marot ring,  $T=\mathsf{T}(R)$, $H=R^{\bullet}$, and $Q=\mathsf{q}(H)$.
\begin{enumerate}
\item If $I\in\mathcal{F}_v(R)$, then $(R:_T I)^{\bullet}=(H:_Q I^{\bullet})$.

\smallskip
\item If $\emptyset\not=\mathfrak{a}\subset Q$ is $H$-fractional, then $(\mathfrak a_{v_R})^{\bullet}= \mathfrak a_{v_H}$.
\end{enumerate}
\end{lemma}

\begin{proof} 1. Let $I\in\mathcal{F}_v(R)$. It follows that $(R:_T I)^{\bullet}=(R:_T (I^{\bullet})_{v_R})^{\bullet}=(R:_T I^{\bullet})^{\bullet}=(H:_Q I^{\bullet})$ by Lemma \ref{2.1}.6 and Lemma \ref{2.3}.3.

\smallskip
2. Let $\emptyset\not=\mathfrak{a}\subset Q$ be $H$-fractional. Then $\mathfrak{a}\subset T$ is regular and $R$-fractional, hence $(R:_T\mathfrak{a})\in\mathcal{F}_v(R)$. By 1. and Lemma \ref{2.3}.3 we have $(a_{v_R})^{\bullet}=(R:_T (R:_T\mathfrak{a}))^{\bullet}=(H:_Q (R:_T\mathfrak{a})^{\bullet})=(H:_Q (H:_Q\mathfrak{a}))=\mathfrak{a}_{v_H}$.
\end{proof}

\medskip
\begin{theorem} \label{3.5}
Let $R$ be a ring and $H=R^{\bullet}$ its monoid of regular elements.
\begin{enumerate}
\item Then  $R$ is a $v$-Marot ring if and only if  the maps
  \begin{align*}
    \iota^\bullet \colon \begin{cases}
      \mathcal F_v (R) &\to \mathcal F_v (H) \setminus \{ \emptyset \} \\
      I &\mapsto I^{\bullet}
    \end{cases}
    \text{ and }
    \iota^\circ \colon \begin{cases}
      \mathcal F_v (H) \setminus \{ \emptyset \} &\to \mathcal F_v (R) \\
      \mathfrak a &\mapsto \mathfrak a_{v_R}
    \end{cases}
  \end{align*}
  are inclusion preserving semigroup isomorphisms which are inverse to each other. If this holds, then
  \begin{enumerate}
    \item $\iota^\bullet | \mathcal F_v (R)^\times \colon \mathcal F_v (R)^\times \to \mathcal F_v (H)^\times$ is a
      group isomorphism.

    \item $\iota^\bullet | \mathcal I_v (R) \colon \mathcal I_v (R) \to \mathcal I_v (H) \setminus \{ \emptyset \}$ is a semigroup isomorphism.

    \item $\iota^\bullet | v$-$\spec(R) \colon v$-$\spec(R) \to v$-$\spec(H) \setminus \{
      \emptyset \}$ is a bijection.
    \end{enumerate}

\smallskip
\item $(\widehat R)^{\bullet}=\widehat{(R^{\bullet})}$. If $R$ is completely integrally closed, then $R^{\bullet}$ is completely integrally closed, and if $R$ is a $v$-Marot ring, then the converse holds.

\smallskip
\item If $R$ is a Mori ring, then $H$ is a Mori monoid, and if $R$ is a $v$-Marot ring, then the converse holds.

\smallskip
\item If $R$ is a Krull ring, then $H$ is a Krull monoid, and if $R$ is a $v$-Marot ring, then the converse holds.
\end{enumerate}
\end{theorem}

\begin{proof}
We set $T = \mathsf T (R)$ and $Q=\mathsf q (H)$.

\smallskip
1. If $\iota^{\bullet}$ and $\iota^{\circ}$ have the mentioned properties, then
\[
I = (\iota^{\circ} \circ \iota^{\bullet}) (I) = \iota^{\circ} (I^{\bullet}) = (I^{\bullet})_{v_R} \quad \text{for every} \ I \in \mathcal F_v (R) \,,
\]
and thus $R$ is a $v$-Marot ring.

Conversely, suppose that $R$ is a $v$-Marot ring.
Let $I,J\in\mathcal{F}_v(R)$ and $\mathfrak{a},\mathfrak{b}\in\mathcal{F}_v(H)\backslash\{\emptyset\}$. Clearly, $I^{\bullet}\not=\emptyset$ and $I^{\bullet}$ is $H$-fractional. Therefore, Lemma \ref{3.4}.2 implies that $(I^{\bullet})_{v_H}=((I^{\bullet})_{v_R})^{\bullet}=I^{\bullet}$, and thus $\iota^{\bullet}(I)\in\mathcal{F}_v(H)\backslash\{\emptyset\}$. Note that $\mathfrak{a}\subset R$ is regular and $R$-fractional, hence $\iota^{\circ}(\mathfrak{a})=\mathfrak{a}_{v_R}\in\mathcal{F}_v(R)$. This implies that $\iota^{\bullet}$ and $\iota^{\circ}$ are well-defined maps. Clearly, both maps $\iota^{\bullet}$ and $\iota^{\circ}$ are inclusion preserving.

\smallskip
Observe that $(\iota^{\bullet}\circ\iota^{\circ})(\mathfrak{a})=(\mathfrak{a}_{v_R})^{\bullet}=\mathfrak{a}_{v_H}=\mathfrak{a}$ by Lemma \ref{3.4}.2. It is obvious that $(\iota^{\circ}\circ\iota^{\bullet})(I)=(I^{\bullet})_{v_R}=I$. Consequently, $\iota^{\bullet}$ and $\iota^{\circ}$ are mutually inverse.

\smallskip
It is clear that $\iota^{\bullet}(R)=R^{\bullet}$ and $\iota^{\circ}(R^{\bullet})=R$. Obviously, $I^{\bullet}J^{\bullet}$ is a non-empty $H$-fractional subset of $Q$. Consequently, Lemmas \ref{2.1}.4 and Lemma \ref{3.4}.2 imply that $\iota^{\bullet}(I\cdot_{v_R} J)=((IJ)_{v_R})^{\bullet}=(((I^{\bullet})_{v_R}(J^{\bullet})_{v_R})_{v_R})^{\bullet}=((I^{\bullet}J^{\bullet})_{v_R})^{\bullet}=(I^{\bullet}J^{\bullet})_{v_H}=\iota^{\bullet}(I)\cdot_{v_H}\iota^{\bullet}(J)$. Note that $\mathfrak{a}\mathfrak{b}$ is a non-empty $H$-fractional subset of $Q$. By Lemma \ref{3.4}.2 we infer that $\iota^{\circ}(\mathfrak{a}\cdot_{v_H}\mathfrak{b})=((\mathfrak{a}\mathfrak{b})_{v_H})_{v_R}=(((\mathfrak{a}\mathfrak{b})_{v_R})^{\bullet})_{v_R}=(\mathfrak{a}\mathfrak{b})_{v_R}=(\mathfrak{a}_{v_R}\mathfrak{b}_{v_R})_{v_R}=\iota^{\circ}(\mathfrak{a})\cdot_{v_R}\iota^{\circ}(\mathfrak{b})$.

\smallskip
It is straightforward to prove that $\iota^{\bullet}(\mathcal{I}_v(R))=\mathcal{I}_v(H)\setminus\{\emptyset\}$ and $\iota^{\bullet}(\mathcal{F}_v(R)^{\times})=\mathcal{F}_v(H)^{\times}$. Therefore, $\iota^{\bullet}|\mathcal{I}_v(R)\colon\mathcal{I}_v(R)\to\mathcal{I}_v(H)\setminus\{\emptyset\}$ and $\iota^{\bullet}|\mathcal{F}_v(R)^{\times}\colon\mathcal{F}_v(R)^{\times}\to\mathcal{F}_v(H)^{\times}$ are semigroup isomorphisms. Consequently, $\iota^{\bullet}|\mathcal{F}_v(R)^{\times}\colon\mathcal{F}_v(R)^{\times}\to\mathcal{F}_v(H)^{\times}$ is a group isomorphism.

\smallskip
It remains to show that $\iota^{\bullet}(v$-$\spec(R))=v$-$\spec(H)\setminus\{\emptyset\}$. First let $P\in v$-$\spec(R)$ and $x,y\in H$ be such that $xy\in\iota^{\bullet}(P)$. Then $xy\in P$, hence $x\in P\cap H=\iota^{\bullet}(P)$ or $y\in P\cap H=\iota^{\bullet}(P)$.

\smallskip
Conversely, let $\mathfrak{p}\in v$-$\spec(H)\setminus\{\emptyset\}$. It is sufficient to prove that $\iota^{\circ}(\mathfrak{p})\in v$-$\spec(R)$. Let $x,y\in R$ be such that $xy\in\iota^{\circ}(\mathfrak{p})$. Let $z\in\mathfrak{p}$. Since $R$ is a $v$-Marot ring, there are some $E,F\subset H$ such that $\{x,z\}_{v_R}=E_{v_R}$ and $\{y,z\}_{v_R}=F_{v_R}$. Note that $\mathfrak{p}_{v_R}$ is an ideal of $R$, hence $\{xy,xz,yz,z^2\}\subset\mathfrak{p}_{v_R}$. This implies that $EF\subset (EF)_{v_R}=(E_{v_R}F_{v_R})_{v_R}=(\{x,z\}_{v_R}\{y,z\}_{v_R})_{v_R}=(\{x,z\}\{y,z\})_{v_R}=\{xy,xz,yz,z^2\}_{v_R}\subset\mathfrak{p}_{v_R}$. Therefore, $EF\subset\mathfrak{p}_{v_R}\cap H=(\mathfrak{p}_{v_R})^{\bullet}=\mathfrak{p}$, and thus $E\subset\mathfrak{p}$ or $F\subset\mathfrak{p}$. It follows that $x\in E_{v_R}\subset\iota^{\circ}(\mathfrak{p})$ or $y\in F_{v_R}\subset\iota^{\circ}(\mathfrak{p})$.

\smallskip
2. By definition, we have
\[
\begin{aligned}
(\widehat R)^{\bullet} & = \widehat R \cap T^{\bullet} \\
 & = \{ x \in T^{\bullet} \mid  \text{there is a} \ c \in R^{\bullet} \ \text{such that} \ cx^n \in R \ \text{for all} \ n \in \N \} \\
 & = \{ x \in T^{\bullet} \mid  \text{there is a} \ c \in R^{\bullet} \ \text{such that} \ cx^n \in R^{\bullet} \ \text{for all} \ n \in \N \} \\
 & = \{ x \in \mathsf{q}(R^{\bullet}) \mid  \text{there is a} \ c \in R^{\bullet} \ \text{such that} \ cx^n \in R^{\bullet} \ \text{for all} \ n \in \N \} = \widehat{(R^{\bullet})} \,.
 \end{aligned}
\]
If $R$ is completely integrally closed, then $R = \widehat R$ and hence $R^{\bullet} = (\widehat R)^{\bullet} = \widehat{(R^{\bullet})}$ and hence $R^{\bullet}$ is completely integrally closed. Suppose that $R$ is a $v$-Marot ring and that $R^{\bullet}$ is completely integrally closed. Let $x\in\widehat{R}$. Then $\{1,x\}_{v_R}\in\mathcal{F}_v(R)$ and $\{1,x\}_{v_R}\subset\widehat{R}$ by Lemma \ref{2.3}.2. Therefore, $\{1,x\}_{v_R}=E_{v_R}$ for some $E\subset T^{\bullet}$. It follows that $E\subset (\widehat{R})^{\bullet}=\widehat{(R^{\bullet})}=R^{\bullet}\subset R$. Consequently, $x\in E_{v_R}\subset R$.

\smallskip
3. If $R$ is a $v$-Marot ring, then 1. implies that $R$ is a Mori ring if and only if $H$ is a Mori monoid.
Now suppose that $R$ is a Mori ring, and let $(\mathfrak{a}_n)_{n \ge 0}$ be an ascending chain of regular divisorial ideals of $H$. Note that $((\mathfrak{a}_n)_{v_R})_{n\ge 0}$ is an ascending chain of regular divisorial ideals of $R$, which becomes stationary by assumption. By Lemma \ref{2.3}.3 it follows that $\mathfrak{a}_n=((\mathfrak{a}_n)_{v_R})^{\bullet}$ for all $n\in\mathbb{N}_{0}$, hence $(\mathfrak{a}_n)_{n\ge 0}$ becomes stationary.

\smallskip
4. This follows immediately from 2. and 3.
\end{proof}

\smallskip
The fact that a Marot ring $R$ is a Krull ring if and only if the monoid $R^{\bullet}$ is a Krull monoid was first proved by Halter-Koch in \cite{HK93f}. We continue with a lemma on class groups.
The arithmetical significance of the distribution of prime divisorial ideals in the classes will be discussed after Corollary \ref{4.4}.

\medskip
\begin{corollary} \label{3.6}
Let $R$ be a ring, $H=R^{\bullet}$, $Q=\mathsf{q}(H)$, $\mathcal{H}_R=\{zR\mid z\in Q\}$, and $\mathcal{H}_H=\{zH\mid z\in Q\}$.
\begin{enumerate}
\item If $\mathcal{F}_v(R)/\mathcal{H}_R$ is finite, then $\mathcal{F}_v(H)/\mathcal{H}_H$ is finite.

\smallskip
\item If $\mathcal{F}_v(R)/\mathcal{H}_R=\{P\mathcal{H}_R\mid P\in v$-$\spec(R)\}$, then $\mathcal{F}_v(H)/\mathcal{H}_H=\{\mathfrak{p}\mathcal{H}_H\mid\mathfrak{p}\in v$-$\spec(H)\}$.

\smallskip
\item If $R$ is a $v$-Marot ring, then there is an isomorphism
\[
\mathcal{C}_v(R)=\mathcal{F}_v(R)^{\times}/\mathcal{H}_R\to\mathcal{F}_v(H)^{\times}/\mathcal{H}_H=\mathcal{}C_v(H)\,,
\]
which maps the set of classes of $\mathcal{C}_v(R)$ containing prime ideals $\mathfrak{p}\in v$-$\spec(R)$ onto the set of classes of $\mathcal{C}_v(H)$ containing prime ideals $\mathfrak{p}\in v$-$\spec(H)$.
\end{enumerate}
\end{corollary}

\begin{proof}
1. Let $\mathcal{F}_v(R)/\mathcal{H}_R$ be finite. Observe that $\{I_{v_R}\mathcal{H}_R\mid I\in\mathcal{F}_v(H)\setminus\{\emptyset\}\}=\{I_{v_R}\mathcal{H}_R\mid I\in E\}$ for some finite $E\subset\mathcal{F}_v(H)\setminus\{\emptyset\}$. It is sufficient to prove that $\mathcal{F}_v(H)/\mathcal{H}_H=\{I\mathcal{H}_H\mid I\in E\cup\{\emptyset\}\}$. Let $J\in\mathcal{F}_v(H)\setminus\{\emptyset\}$. Then $J_{v_R}=cI_{v_R}$ for some $c\in Q$ and $I\in E$. Using Lemma \ref{2.3}.3 this implies that $J=(J_{v_R})^{\bullet}=(cI_{v_R})^{\bullet}=((cI)_{v_R})^{\bullet}=cI$. Therefore, $J\mathcal{H}_H=I\mathcal{H}_H$.

\smallskip
2. Let $\mathcal{F}_v(R)/\mathcal{H}_R=\{P\mathcal{H}_R\mid P\in v$-$\spec(R)\}$ and $I\in\mathcal{F}_v(H)$. Without restriction let $I\not=\emptyset$. Note that $I_{v_R}\in\mathcal{F}_v(R)$, and thus there is some $P\in v$-$\spec(R)$ such that $P\mathcal{H}_R=I_{v_R}\mathcal{H}_R$. Therefore, $P=cI_{v_R}$ for some $c\in Q$. Since $cI\in\mathcal{F}_v(H)$, it follows from Lemma \ref{2.3}.3 that $P^{\bullet}=(cI_{v_R})^{\bullet}=((cI)_{v_R})^{\bullet}=cI$. Using this it is straightforward to prove that $P^{\bullet}\in v$-$\spec(H)$. Therefore, $I\mathcal{H}_H=P^{\bullet}\mathcal{H}_H\in\{\mathfrak{p}\mathcal{H}_H\mid\mathfrak{p}\in v$-$\spec(H)\}$.

\smallskip
3. This follows easily from Theorem \ref{3.5}.1.
\end{proof}

\smallskip
In view of the Corollary \ref{3.6}, we identify $\mathcal{C}_v(R)$ and $\mathcal{C}_v(H)$ in the case of $v$-Marot rings. In the Krull setting, we write $\mathcal{C}(R)$ and $\mathcal{C}(H)$ for the $v$-class groups, and these are isomorphic to the respective divisor class groups.

\medskip
\begin{proposition} \label{3.7}
Let $R$ be a ring, $T=\mathsf{T}(R)$, and $\mathfrak{f}=(R:_T\widehat{R})$.
\begin{enumerate}
\item Let $I\in\mathcal{F}_v(R)$. Then $I$ is $v$-invertible if and only if $(I:_T I)=R$.

\smallskip
\item $\widehat{R}=\bigcup_{I\in\mathcal{F}_v(R)} (I:_T I)$ and if $\mathfrak{f}$ is regular, then $\widehat{R}\in\mathcal{F}_v(R)$.

\smallskip
\item $R$ is completely integrally closed if and only if $\mathcal{F}_v(R)=\mathcal{F}_v(R)^{\times}$.

\smallskip
\item If $R$ is completely integrally closed and  $\mathcal{C}_v(R)$ is finite, then $\mathcal{C}_v(R^{\bullet})$ is finite.
\end{enumerate}
\end{proposition}

\begin{proof}
1. Using Lemma \ref{2.1}.3 we obtain that $I$ is $v$-invertible if and only if $(II^{-1})_v=R$ if and only if $(II^{-1})^{-1}=R$ if and only if $(I_v:_T I)=R$ if and only if $(I:_T I)=R$.

\smallskip
2. Let $x\in\widehat{R}$ be given. Then there is some $c\in R^{\bullet}$ such that $cx^n\in R$ for all $n\in\mathbb{N}$. Set $I=\{x^n\mid n\in\mathbb{N}_0\}_v$. Note that $\{x^n\mid n\in\mathbb{N}_0\}$ is a regular $R$-fractional subset of $T$, hence $I\in\mathcal{F}_v(R)$. Moreover, $xI\subset\{x^n\mid n\in\mathbb{N}\}_v\subset I$ by Lemma \ref{2.1}.5, and thus $x\in (I:_T I)$. Conversely, let $I\in\mathcal{F}_v(R)$ and $x\in (I:_T I)$. There are some $c,d\in R^{\bullet}$ such that $cI\subset R$ and $d\in cI\cap R^{\bullet}$. Note that $x^n\in (I:_T I)$ for all $n\in\mathbb{N}$, hence $dx^n\in cx^nI\subset cI\subset R$ for all $n\in\mathbb{N}$. Consequently, $x\in\widehat{R}$.\\
Now let $\mathfrak{f}$ be regular. It follows from Lemma \ref{2.1}.1 that $\mathfrak{f}\in\mathcal{I}_v(R)$. Therefore, Lemma \ref{2.3}.2 implies that $(\widehat{R})_{v_R}=(\mathfrak{f}:_T\mathfrak{f})\subset\widehat{R}$, and thus $(\widehat{R})_{v_R}=\widehat{R}$. Consequently, $\widehat{R}\in\mathcal{F}_v(R)$.

\smallskip
3. Clearly, $R\subset (I:_T I)$ for all $I\in\mathcal{F}_v(R)$. Therefore, 1. and 2. imply that $R$ is completely integrally closed if and only if $(I:_T I)=R$ for all $I\in\mathcal{F}_v(R)$ if and only if $\mathcal{F}_v(R)=\mathcal{F}_v(R)^{\times}$.

\smallskip
4. This is an immediate consequence of Theorem \ref{3.5}.2, Corollary \ref{3.6}.1, and of 3.
\end{proof}

\medskip
\begin{proposition} \label{3.8}
Let $R$ be a ring and $R \subset S \subset \mathsf T (R)$ an overring of $R$ such that $S\in\mathcal{F}_v(R)$.
\begin{enumerate}
\item $\mathcal{F}_v(S)\subset\mathcal{F}_v(R)$.

\smallskip
\item If $R$ is a Mori ring, then $S$ is a Mori ring.

\smallskip
\item If $R$ is a $v$-Marot ring, then $S$ is a $v$-Marot ring.

\smallskip
\item If $R$ is a Mori ring such that $(R:\widehat{R})$ is regular, then $\widehat{R}$ is a Krull ring.
\end{enumerate}
\end{proposition}

\begin{proof}
We set $T = \mathsf T (R)$.

1. Let $I\in\mathcal{F}_v(S)$. We have $I_{v_R}\subset (I_{v_R})_{v_S}=I_{v_S}=I\subset I_{v_R}$ by Lemma \ref{2.1}.6, and thus $I_{v_R}=I$. Moreover, there are some $c,d\in S^{\bullet}$ such that $cI\subset S$ and $dS\subset R$. Observe that $cd\in R^{\bullet}$, $cdI\subset R$ and $I^{\bullet}\not=\emptyset$. Therefore, $I\in\mathcal{F}_v(R)$.

\smallskip
2. Let $(I_i)_{i\in\mathbb{N}}$ be an ascending chain of elements of $\mathcal{I}_v(S)$. There is some $c\in (R:_T S)^{\bullet}$. Obviously, $(cI_i)_{i\in\mathbb{N}}$ is an ascending chain of elements of $\mathcal{I}_v(S)$ such that $cI_i\subset R$ for all $i\in\mathbb{N}$. It follows by 1. that $(cI_i)_{i\in\mathbb{N}}$ is an ascending chain of elements of $\mathcal{I}_v(R)$, hence there is some $k\in\mathbb{N}$ such that $cI_i=cI_k$ for all $i\in\mathbb{N}_{\geq k}$. This immediately implies that $I_i=I_k$ for all $i\in\mathbb{N}_{\geq k}$.

\smallskip
3. Let $R$ be a $v$-Marot ring and $I\in\mathcal{F}_v(S)$. By 1. we have $I\in\mathcal{F}_v(R)$, hence $I=(I^{\bullet})_{v_R}$. It follows that $I=I_{v_S}=((I^{\bullet})_{v_R})_{v_S}=(I^{\bullet})_{v_S}$ by Lemma \ref{2.1}.6.

\smallskip
4. Let $f \in (R:_T\widehat{R})^{\bullet}$ and $x\in\widehat{\widehat{R}}$. Then there is a $c\in (\widehat{R})^{\bullet}$ such that $cx^n\in\widehat{R}$ for all $n\in\N$. Then $fc\in R^{\bullet}$ and $fcx^n\in R$ for all $n\in\N$. This implies that $x\in\widehat{R}$. Consequently, $\widehat{R}$ is completely integrally closed. It follows from Proposition \ref{3.7}.2 that $\widehat{R}\in\mathcal{F}_v(R)$, and hence 2. implies that $\widehat{R}$ is a Mori ring.
\end{proof}

\medskip
\begin{proposition} \label{3.11}
Let $R$ be a Mori ring, $T=\mathsf{T}(R)$, and $S\subset R^{\bullet}$ a multiplicatively closed subset.
\begin{enumerate}
\item If $E\subset T$ is regular and $R$-fractional, then $S^{-1}(R:_T E)=(S^{-1}R:_T S^{-1}E)$ and $S^{-1}E_v=(S^{-1}E)_{v_{S^{-1}R}}$.

\smallskip
\item If $I\in\mathcal{F}_v(S^{-1}R)$, then $I=S^{-1}J$ for some $J\in\mathcal{F}_v(R)$.

\smallskip
\item $\widehat{S^{-1}R}=S^{-1}\widehat{R}$.

\smallskip
\item If $(R:_T\widehat{R})$ is regular, then $S^{-1}(R:_T\widehat{R})=(S^{-1}R:_T\widehat{S^{-1}R})$.

\smallskip
\item If $R$ is a $v$-Marot ring, then $S^{-1}R$ is a $v$-Marot ring.

\smallskip
\item If $\mathcal{F}_v(R)/\{aR\mid a\in T^{\bullet}\}$ is finite, then $\mathcal{F}_v(S^{-1}R)/\{aS^{-1}R\mid a\in T^{\bullet}\}$ is finite.
\end{enumerate}
\end{proposition}

\begin{proof}
1. Let $E\subset T$ be regular and $R$-fractional, and let $x\in S^{-1}(R:_T E)$. Then $xs\in (R:_T E)$ for some $s\in S$. Let $z\in S^{-1}E$. Then $zt\in E$ for some $t\in S$. We have $stxz\in R$, hence $xz\in S^{-1}R$, and thus $x\in (S^{-1}R:_T S^{-1}E)$. Conversely, let $x\in (S^{-1}R:_T S^{-1}E)$. Then $xE\subset S^{-1}R$. Note that $E_v\in\mathcal{F}_v(R)$, hence there is some finite regular $F\subset E$ such that $F_v=E_v$. We obtain that $xF\subset S^{-1}R$. This implies that $xtF\subset R$ for some $t\in S$. It follows that $xtE\subset xtE_v=xtF_v\subset (xtF)_v\subset R$. Therefore, $xt\in (R:_T E)$, and thus $x\in S^{-1}(R:_T E)$.\\
Consequently, $S^{-1}E_v=S^{-1}(R:_T (R:_T E))=(S^{-1}R:_T S^{-1}(R:_T E))=(S^{-1}R:_T (S^{-1}R:_T S^{-1}E))=(S^{-1}E)_{v_{S^{-1}R}}$.

\smallskip
2. Let $I\in\mathcal{F}_v(S^{-1}R)$. Then $cI\in\mathcal{I}_v(S^{-1}R)$ for some $c\in T^{\bullet}$. Set $J=cI\cap R$. Observe that $cI=S^{-1}J$ and $J$ is a regular subset of $R$. Let $x\in J_v$. Then $x(R:_T J)\subset R$. Therefore, 1. implies that $x(S^{-1}R:_T cI)=x(S^{-1}R:_T S^{-1}J)=xS^{-1}(R:_T J)\subset S^{-1}R$. Consequently, $x\in (cI)_{v_{S^{-1}R}}\cap R=J$. This implies that $J\in\mathcal{I}_v(R)$, hence $I=S^{-1}(c^{-1}J)$ and $c^{-1}J\in\mathcal{F}_v(R)$.

\smallskip
3.  Let $x\in\widehat{S^{-1}R}$ be given. It follows by Proposition \ref{3.7}.2 that $x\in (I:_T I)$ for some $I\in\mathcal{F}_v(S^{-1}R)$. There is some $J\in\mathcal{F}_v(R)$ such that $I=S^{-1}J$ by 2. We infer by 1. and Proposition \ref{3.7}.2 that $x\in S^{-1}(J:_T J)\subset S^{-1}\widehat{R}$. Conversely, let $x\in S^{-1}\widehat{R}$. Then $xt\in\widehat{R}$ for some $t\in S$. Consequently, there is some $c\in R^{\bullet}$ such that $ct^nx^n\in R$ for all $n\in\mathbb{N}$. We obtain that $c\in (S^{-1}R)^{\bullet}$ and $cx^n\in S^{-1}R$ for all $n\in\mathbb{N}$. Therefore, $x\in\widehat{S^{-1}R}$.

\smallskip
4. Let $(R:_T\widehat{R})$ be regular. Then $\widehat{R}$ is regular and $R$-fractional. Consequently, $S^{-1}(R:_T\widehat{R})=(S^{-1}R:_T S^{-1}\widehat{R})=(S^{-1}R:_T\widehat{S^{-1}R})$ by 1. and 3.

\smallskip
5. Let $R$ be a $v$-Marot ring and $I\in\mathcal{F}_v(S^{-1}R)$. It follows by 2. that $I=S^{-1}J$ for some $J\in\mathcal{F}_v(R)$. By 1. we have $I=S^{-1}(J^{\bullet})_v=(S^{-1}J^{\bullet})_{v_{S^{-1}R}}$. Since $S^{-1}J^{\bullet}\subset I^{\bullet}$ we infer that $I=(I^{\bullet})_{v_{S^{-1}R}}$.

\smallskip
6. It follows from 1. and 2. that $f:\mathcal{F}_v(R)\rightarrow\mathcal{F}_v(S^{-1}R)$ defined by $f(I)=S^{-1}I$ for all $I\in\mathcal{F}_v(R)$ is a surjective map. Using this it is straightforward to prove that $\overline{f}:\mathcal{F}_v(R)/\{aR\mid a\in T^{\bullet}\}\rightarrow\mathcal{F}_v(S^{-1}R)/\{aS^{-1}R\mid a\in T^{\bullet}\}$ defined by $\overline{f}(I\{aR\mid a\in T^{\bullet}\})=f(I)\{aS^{-1}R\mid a\in T^{\bullet}\}$ is a surjective map. This immediately implies the assertion.
\end{proof}

\bigskip
\section{C-monoids and C-rings} \label{4}
\bigskip

In this section we define C-rings as commutative rings whose monoid of regular elements is a C-monoid. Originally, C-monoids and C-domains have been introduced in order to study the arithmetic of non-integrally closed higher-dimensional  noetherian domains, and since then their arithmetic has been studied in detail. In Proposition \ref{4.2} we summarize some arithmetical finiteness results for C-monoids. The main result in this section is Theorem \ref{4.8} (together with Corollary \ref{4.9}). It states that a $v$-Marot Mori ring $R$ with regular conductor $\mathfrak f = (R : \widehat R)$, for which the residue class ring $R/\mathfrak f$ and the class group $\mathcal C (\widehat R)$ are both finite, is a C-ring. In particular, this implies that all arithmetical finiteness results of Proposition \ref{4.2} are satisfied.

In order to give the definition of C-monoids we need to recall the concept of class semigroups which are a refinement of ordinary class groups in commutative algebra (a detailed presentation can be found in \cite[Chapter 2]{Ge-HK06a}).
Let $F$ be a monoid and $H \subset F$ a submonoid. Two elements $y, y' \in F$
are called $H$-equivalent, if $y^{-1}H \cap F = {y'}^{-1} H \cap
F$. $H$-equivalence is a congruence relation on $F$. For $y \in
F$, let $[y]_H^F$ denote the congruence class of $y$, and let
\[
\mathcal C (H,F) = \{ [y]_H^F \mid y \in F \} \quad \text{and}
\quad \mathcal C^* (H,F) = \{ [y]_H^F \mid y \in (F \setminus
F^{\times}) \cup \{1\} \}.
\]
Then $\mathcal C (H,F)$ is a semigroup with unit element $[1]_H^F$
(called the {\it class semigroup} of $H$ in $F$) and $\mathcal C^* (H,F)
\subset \mathcal C (H,F)$ is a subsemigroup (called the reduced
class semigroup of $H$ in $F$). It follows from the very definitions that, for every subset $T \subset F$, there is a bijective map
\begin{equation} \label{bijection}
\{ [y]_H^F \mid y \in T\} \to \{ y^{-1}H \cap F \mid y \in T\}, \quad \text{given by} \ [y]_H^F \mapsto y^{-1}H \cap F
\end{equation}
If $\mathcal C (H,F)$ is a torsion
group, then $H \subset F$ is saturated and cofinal, and if $H
\subset F$ is saturated and cofinal, then $\mathcal C (H, F) = \mathsf q (F)/\mathsf q (H)$.

\medskip
\begin{definition} \label{4.1}~

\begin{enumerate}
\item Let $H$ be a monoid. If $H$ is a submonoid of a factorial monoid $F$
      such that $H \cap F^\times = H^\times$ and $\mathcal{C}^*(H,F)$ is finite,
      then $H$ is called a {\it {\rm C}-monoid}.

\smallskip
\item A ring $R$ is called a \textit{{\rm C}-ring} if
      $R^{\bullet}$ is a {\rm C}-monoid.
\end{enumerate}
\end{definition}

\smallskip
Next we gather some arithmetical concepts which are required to present the main arithmetical finiteness results of C-monoids.
Let $H$ be a Mori monoid. Then every non-unit $a \in H$ can be written as a finite product of irreducible elements, say $a = u_1 \cdot \ldots \cdot u_k$, where $u_1, \ldots, u_k \in H$ are irreducible and $k \in \N$ is called the length of the factorization. Since $H$ is a Mori monoid, the set $\mathsf L (a) \subset \N$ of all possible factorization lengths is finite (\cite[Theorem 2.2.9]{Ge-HK06a}), and $\mathsf L (a)$ is called the {\it set of lengths} of $a$. It is convenient to set $\mathsf L (a) = \{0\}$ for $a \in H^{\times}$. For $k \in \N$, the set $\mathcal U_k (H)$ denotes the {\it union of sets of lengths} $\mathsf L (a)$ (over all $a \in H$) with $k \in \mathsf L (a)$. Unions of sets of lengths can either be finite or infinite. For a finite set $L = \{m_1, \ldots, m_k \} \subset \Z$ with $k \in \N$ and $m_1 < \ldots < m_k$, we denote by $\Delta (L)$ the set of (successive) distances of $L$, that is $\Delta (L) = \{ m_{\nu+1} - m_{\nu} \mid \nu \in [1, k-1] \}$. Then
\[
\Delta (H) = \bigcup_{a \in H} \Delta \big( \mathsf L (a) \big)
\]
denotes the {\it set of distances} of $H$. By definition, $\Delta (H) = \emptyset$ if and only if $|\mathsf L (a) |=1$ for all $a \in H$ if and only if $\mathcal U_k (H) = \{k\}$ for all $k \in \N$. If there is some $a \in H$ with $|\mathsf L (a)| > 1$, then $|\mathsf L (a^n)| > n$ for every $n \in \N$.

In C-monoids sets of lengths and their unions have a well-defined structure. We do not repeat here the rather involved definitions of AAMPs (almost arithmetical multiprogression), AAPs (almost arithmetical progression), or of the catenary degree (they can be found in \cite{Ge-HK06a}). Roughly speaking, AAMPs and AAPs are generalized arithmetical progressions which are controlled by several parameters, and Proposition \ref{4.2} states that all these parameters are globally bounded.

\medskip
\begin{proposition}[\bf Arithmetic Properties] \label{4.2}
Let $H$ be a {\rm C}-monoid.
\begin{enumerate}
\item $H$ has finite catenary degree and finite set of distances $\Delta (H)$.

\smallskip
\item There is a constant $M^* \in \N$ such that for every $a \in H$ the set of lengths $\mathsf L (a)$ is an {\rm AAMP}  with difference $d \in \Delta (H)$ and bound $M^*$.

\smallskip
\item There are constants $k^*, M^* \in \N$ such that for all $k \ge k^*$ the union of sets of lengths $\mathcal U_k (H)$ is an {\rm AAP} with difference $d$ and bound $M^*$.
\end{enumerate}
\end{proposition}

\begin{proof}
1. and 2. can be found in \cite[Theorems 3.3.4 and 4.6.6]{Ge-HK06a}. For 3. see \cite[Theorems 3.10 and 4.2]{Ga-Ge09b}.
\end{proof}

\smallskip
More on the arithmetic of C-monoids can be found in \cite{Fo-Ge05, Fo-Ha06a}. We switch to their  main algebraic properties.

\medskip
\begin{proposition}[\bf Algebraic Properties] \label{4.3}
Let $H$ be a {\rm C}-monoid.
Then $H$ is a Mori monoid with $(H \DP \widehat H) \ne \emptyset$, and the complete integral closure $\widehat H$ is a Krull monoid with finite class group $\mathcal C (\widehat H)$.
\end{proposition}

\begin{proof}
See \cite[Theorems 2.9.11 and 2.9.13]{Ge-HK06a}.
\end{proof}

\smallskip
\begin{corollary} \label{4.4}
Let $R$ be a ring and $H=R^{\bullet}$.
\begin{enumerate}
\item Let $R$ be a Krull ring. If every class of $\mathcal C (R)$ contains a prime divisorial ideal, then every class of $\mathcal C (H)$ contains a prime divisorial ideal. If $\mathcal C (R)$ is finite, then $\mathcal C (H)$ is finite and $R$ is a {\rm C}-ring.

\smallskip
\item Let $R$ be a $v$-Marot {\rm C}-ring. Then $R$ is a Mori ring,  $(R \DP \widehat R)$ is regular, and $\widehat R$ is a $v$-Marot Krull ring with finite class group.
\end{enumerate}
\end{corollary}

\begin{proof}
1. The first statement follows from Corollary \ref{3.6}.2 and from Proposition \ref{3.7}.3. Let $R$ be a Krull ring with finite class group. Then $H$ is a Krull monoid by Theorem \ref{3.5}, and $H$ has finite class group by Proposition \ref{3.7}.4. Thus $H$ is a C-monoid by \cite[Theorem 2.9.12]{Ge-HK06a}.

\smallskip
2. Let $T$ be the total quotient ring  of $R$. It follows from Proposition \ref{4.3} that $H$ is a Mori monoid, $(H:_{T^{\bullet}}\widehat{H})\not=\emptyset$, and  that $\widehat H$ is Krull with finite class group $\mathcal{C}_v(\widehat{H})$. By Theorem \ref{3.5}.2, $R$ is a Mori ring and $(\widehat{R})^{\bullet}=\widehat{H}$. Proposition \ref{3.8} implies that $\widehat R$ is a $v$-Marot Krull ring,  and by Corollary \ref{3.6}.3 we obtain that $\mathcal{C}_v(\widehat{R})\cong\mathcal{C}_v((\widehat{R})^{\bullet})=\mathcal{C}_v(\widehat{H})$ is finite. In order to verify that $(R \DP \widehat R)$ is regular, we choose an element  $x\in (H:_{T^{\bullet}}\widehat{H})$ and have to  show that $x\in (R:_T\widehat{R})$. Let $y\in\widehat{R}$. By Lemma \ref{2.3}.2 it follows that $\{1,y\}_{v_R}\subset\widehat{R}$. Moreover, $\{1,y\}_{v_R}\in\mathcal{F}_v(R)$, hence there is some $E\subset T^{\bullet}$ such that $\{1,y\}_{v_R}=E_{v_R}$. Observe that $E\subset (\widehat{R})^{\bullet}=\widehat{H}$, and thus $xE\subset R^{\bullet}\subset R$. Therefore, $xy\in xE_{v_R}=(xE)_{v_R}\subset R$.
\end{proof}

\smallskip
Let $H$ be a Krull monoid with class group $G$ and let $G_P \subset G$ denote the set of classes containing prime divisorial ideals. Then there is a transfer homomorphism $\theta \colon H \to \mathcal B (G_P)$, where $\mathcal B (G_P)$ is the monoid of zero-sum sequences over $G_P$. Transfer homomorphisms preserve sets of lengths (and other arithmetical invariants), and they allow to show that  the finiteness of $G_P$ induces arithmetical finiteness results (including all what is mentioned in Proposition \ref{4.2}). If moreover  $G_P = G$ is finite, then there is a variety of most precise arithmetical information (see \cite{Ge-HK06a} and the surveys \cite{Ge09a, Sc09b}). Corollary \ref{4.4} reveals that all these results apply to Krull rings and, by Corollary \ref{3.6}.3, the connection is most close in the $v$-Marot case, since in that case  there is an isomorphism between the class groups and a bijection between the set of classes containing prime divisorial ideals. In particular, we regain a classical result by Anderson and Markanda (\cite[Theorem 3.6]{An-Ma85a} and \cite{An-Ma85b}), saying that $R^{\bullet}$ is factorial if and only if $R$ is a Krull ring with trivial class group. We discuss one example.

\medskip
\begin{example} \label{4.5}
Let $R$ be a ring. Then $R[X]$ is a Krull ring  if and only if $R$ is a finite direct product of Krull domains, say $R = R_1 \times \ldots \times R_n$ (\cite[Theorem 3]{Ch01a}). Suppose this holds true. For every Krull domain $R_{\nu }$, with $\nu \in [1,n]$, we have $\mathcal C (R_{\nu}) \cong \mathcal C (R_{\nu}[X])$, and every class of $\mathcal C (R_{\nu}[X])$ contains a prime divisorial ideal (\cite[Theorem 14.3]{Fo73}). Moreover,  $R[X] \cong R_1[X] \times \ldots \times R_n[X]$ is a finite direct product of Krull domains again. For each two Krull monoids $H_1$ and $H_2$, there is an isomorphism  $f \colon \mathcal C (H_1) \oplus \mathcal C (H_2) \to \mathcal C (H_1 \times H_2)$ and if, for $\nu \in [1,2]$, $G_{\nu} \subset \mathcal C (H_{\nu})$ is the set of classes of $\mathcal C (H_{\nu})$ containing prime divisorial ideals, then $f (G_1 \uplus G_2) \subset \mathcal C (H_1 \times H_2)$ is the set of classes containing prime divisorial ideals. Thus, if $R$ has finite class group, then $R[X]$ is a Krull ring with finite class group. Using induction we infer that, for each $s \in \N$,  $R[X_1, \ldots, X_s]$ is a Krull ring with finite class group and hence a C-ring. Moreover, $R[X_1, \ldots, X_s]$ has class group $G = G_1 \oplus \ldots \oplus G_n$ and $G_1 \cup \ldots \cup G_n \subset G$ is the set of classes containing prime divisors, where $G_{\nu}$ is isomorphic to the class group of $R_{\nu}$ for each $\nu \in [1, n]$.
\end{example}

\smallskip
Next we handle finite direct products.

\medskip
\begin{lemma} \label{4.6}
For $\nu \in [1, 2]$, let $R_{\nu}$ be a ring with total quotient ring $T_{\nu}$ and $I_{\nu} \subset T_{\nu}$ a non-empty subset.
\begin{enumerate}
\item $({R_1}\times {R_2}:_{T_1\times T_2} {I_1}\times I_2)=({R_1}:_{T_1} {I_1})\times ({R_2}:_{T_2} I_2)$.

\smallskip
\item  $({I_1}\times I_2)_{v_{{R_1}\times {R_2}}}= ({I_1})_{v_{R_1}}\times ({I_2})_{v_{R_2}}$.
\end{enumerate}
\end{lemma}

\begin{proof} Observe that $T_1\times T_2$ is a total quotient ring of ${R_1}\times {R_2}$.

\smallskip
1. Let  $x = (x_1, x_2) \in T_1\times T_2$, and note that $x_1 I_1 \not=\emptyset$ and $x_2 I_2 \not=\emptyset$. Therefore, $x \in ({R_1}\times {R_2}:_{T_1\times T_2} I_1\times I_2)$ if and only if $(x_1,x_2)(I_1 \times I_2)\subset {R_1}\times {R_2}$ if and only if $x_1 I_1 \times x_2 I_2\subset {R_1}\times {R_2}$ if and only if $x_1 I_1 \subset {R_1}$ and $x_2 I_2 \subset {R_2}$ if and only if $x_1 \in ({R_1}:_{T_1} I_1)$ and $x_2 \in ({R_2}:_{T_2} I_2)$ if and only if $x\in ({R_1}:_{T_1} I_1 )\times ({R_2}:_{T_2} I_2)$.

\smallskip
2.  Note that $({R_1}:_{T_1} I_1)\not=\emptyset$ and $({R_2}:_{T_2} I_2)\not=\emptyset$. By 1. we have $(I_1\times I_2)_{v_{{R_1}\times {R_2}}}=({R_1}\times {R_2}:_{T_1\times T_2} ({R_1}\times {R_2}:_{T_1\times T_2} I_1\times I_2))=({R_1}\times {R_2}:_{T_1\times T_2} ({R_1}:_{T_1} I_1)\times ({R_2}:_{T_2} I_2))=({R_1}:_{T_1} ({R_1}:_{T_1} I_1))\times ({R_2}:_{T_2} ({R_2}:_{T_2} I_2))= ({I_1})_{v_{R_1}}\times ({I_2})_{v_{R_2}}$.
\end{proof}

\medskip
\begin{proposition} \label{4.7}
Let $R_1, R_2$, and $R$ be rings such that $R = R_1 \times R_2$.
\begin{enumerate}
\item $R$ is a Marot ring $($a $v$-Marot ring$)$ if and only if $R_1$ and $R_2$ are Marot rings $(v$-Marot rings$)$.

\smallskip
\item $R$ is a
{\rm C}-ring if and only if the following three conditions are satisfied{\rm \,:}
\begin{enumerate}
\smallskip
\item[(a)] $R_1$ and $R_2$ are both  {\rm C}-rings.

\smallskip

\item[(b)] ${R_1}^{\bullet} = R_1^\times$, \ or \ $\widehat {R_2}{}^\times \negthinspace /R_2^\times$ \ is finite.

\smallskip

\item[(c)] ${R_2}^{\bullet} = R_2^\times$, \ or \ $\widehat {R_1}{}^\times \negthinspace /R_1^\times$ \ is finite.
\end{enumerate}
\end{enumerate}
\end{proposition}

\begin{proof} Clearly, we have $(R_1 \times R_2)^{\bullet}= {R_1}^{\bullet}\times {R_2}^{\bullet}$, and a subset $I \subset R_1 \times R_2$  is an ideal of $R_1 \times R_2$ if and only if $I = I_1 \times I_2$ for some ideal $I_1$ of $R_1$ and some ideal $I_2$ of $R_2$. Furthermore, for all subsets $J_1 \subset R_1$ and $J_2 \subset R_2$, we have $({J_1}\times J_2)^{\bullet}={J_1}^{\bullet}\times {J_2}^{\bullet}$.

\smallskip
1. The fact that $R_1 \times R_2$ is a Marot ring if and only if both $R_1$ and $R_2$ are Marot rings is easy and well-known (see \cite[Proposition 4]{Po-Sp83a}). We verify the statement for $v$-Marot rings. Let ${R_1}\times {R_2}$ be a $v$-Marot ring and $I_1 \in\mathcal{I}_v({R_1})$. It follows from Lemma \ref{4.6} that $I_1 \times {R_2}\in\mathcal{I}_v({R_1}\times {R_2})$. Therefore, Lemma \ref{4.6} implies that $I_1 \times {R_2}=((I_1 \times {R_2})^{\bullet})_{v_{{R_1}\times {R_2}}}=({I_1}^{\bullet}\times {R_2}^{\bullet})_{v_{{R_1}\times {R_2}}}=({I_1}^{\bullet})_{v_{R_1}}\times ({R_2}^{\bullet})_{v_{R_2}}=({I_1}^{\bullet})_{v_{R_1}}\times {R_2}$, and thus $I_1 =({I_1}^{\bullet})_{v_{R_1}}$. This implies that ${R_1}$ is a $v$-Marot ring. Analogously, it follows that ${R_2}$ is a $v$-Marot ring.

Conversely, suppose that ${R_1}$ and ${R_2}$ are $v$-Marot rings and $I \in\mathcal{I}_v({R_1}\times {R_2})$. Thus $I = I_1 \times I_2$ for some ideal $I_1$ of ${R_1}$ and some ideal $I_2$ of ${R_2}$. By Lemma \ref{4.6}.2, we infer that $I_1 \times I_2 = I = I_{v_{{R_1}\times {R_2}}}= ({I_1})_{v_{R_1}}\times ({I_2})_{v_{R_2}}$, and since $I_1 \not=\emptyset$ and $I_2 \not=\emptyset$ it follows that $I_1 = ({I_1})_{v_{R_1}}$ and $I_2 = ({I_2})_{v_{R_2}}$. Thus we obtain that ${I_1}^{\bullet}\times {I_2}^{\bullet}=I^{\bullet}\not=\emptyset$, and thus ${I_1}^{\bullet}\not=\emptyset$ and ${I_2}^{\bullet}\not=\emptyset$. Therefore, $I_1 \in\mathcal{I}_v({R_1})$ and $I_2 \in\mathcal{I}_v({R_2})$. It follows from Lemma \ref{4.6} that $I =I_1 \times  I_2 = ({I_1}^{\bullet})_{v_{R_1}}\times ({I_2}^{\bullet})_{v_{R_2}}=({I_1}^{\bullet}\times {I_2}^{\bullet})_{v_{{R_1}\times {R_2}}}=(I^{\bullet})_{v_{{R_1}\times {R_2}}}$.

\smallskip
2. The characterization of when $R_1 \times R_2$ is a C-ring follows from \cite[Theorem 2.9.16]{Ge-HK06a}.
\end{proof}

\medskip
Corollary \ref{4.4}.2 provides a list of necessary conditions for a $v$-Marot ring to be a C-ring. The next theorem demonstrates that these necessary conditions together with one additional condition (namely the finiteness of the residue class ring) actually guarantees that a $v$-Marot ring is a C-ring. In the domain case there are ring theoretical characterizations of C-domains in various settings. Among others, if $R$ is a non-local semilocal noetherian domain, then $R$ is a C-domain if and only if the class group $\mathcal C (\widehat R)$ and the residue class ring $R / (R : \widehat R)$ are both finite (\cite[Corollary 4.5]{Re13a}), and these are precisely the conditions which we have in Theorem \ref{4.8}.2.

\medskip
\begin{theorem} \label{4.8}
Let $R$ be a $v$-Marot Mori ring such that  $\mathfrak{f}=(R \DP {\widehat R})$ is regular.
\begin{enumerate}
\item ${\widehat R}$ is a $v$-Marot Krull ring.

\smallskip
\item If $\mathcal{C}({\widehat R})$ and ${\widehat R}/\mathfrak{f}$ are both finite,
      then $R$ is a {\rm C}-ring.
\end{enumerate}
\end{theorem}

\begin{proof}
1. This follows from Proposition \ref{3.7}.2 and from Proposition \ref{3.8}.

\smallskip
2. We use the following facts from the theory of Krull monoids (see \cite[Section 2.4]{Ge-HK06a}). Every Krull monoid $H$ has a divisor theory, and $H = H^{\times} \times H_0$ for a reduced Krull monoid $H_0$ with $H_0 \cong H_{\red}$. If $\varphi \colon H \to F = \mathcal F (P)$ is a divisor theory, then there is an isomorphism $\varphi^* \colon F \to \mathcal I_v^* (H)$ and an isomorphism $\overline{\varphi} \colon G = \mathsf q (F)/ \mathsf q ( \varphi (H)) \to \mathcal C_v (H)$. We consider the classes $g \in G$ as subsets of $\mathsf q (F)$. Thus for any subset $F' \subset F$, the set $g \cap F'$ consists of all  elements of $F'$ lying in the class $g$.

By Theorem \ref {3.5} and by 1., ${\widehat R}^{\bullet}$ is a Krull monoid. Thus there exists a reduced Krull monoid $D$ such that ${\widehat R}^{\bullet} = {\widehat R}^\times \times D$ and $D \cong ({\widehat R}^{\bullet})_{\red}$. Let the embedding $D \hookrightarrow F_0 = \mathcal{F}(P)$ be a divisor theory,  and set $F = {\widehat R}^{\times} \times F_0$. Then
$\RingReg{R} \subset {\widehat R}^{\bullet} = {\widehat R}^\times \times D \subset {\widehat R}^\times \times F_0 = F$ and $F$ is factorial. Also $F^\times = {\widehat R}^\times$ and
$F^\times \cap \RingReg{R} = {\widehat R}^\times \cap \RingReg{R} = R^\times$.

Thus we have that $\RingReg{R}$ is a submonoid of a factorial monoid $F$
and it remains to prove that the reduced class semigroup $\mathcal{C}^*(\RingReg{R},F)$ is finite. We will even show that the class semigroup $\mathcal C (R^{\bullet}, F)$ is finite, and this will be done in several steps.

Let $\eta \colon F \to F_0$ denote the canonical projection. Then $\eta |
      {\widehat R}^{\bullet} \colon {\widehat R}^{\bullet} \to F_0$ is a divisor theory, and there is an isomorphism  $\eta^* \colon F_0 \to \mathcal I_v^* ({\widehat R}^{\bullet})$. In particular, if $a,b \in F$ and $a F \cap {\widehat R}^{\bullet}
      = b F \cap {\widehat R}^{\bullet}$, then $aF = bF$, and if $a \in {\widehat R}^{\bullet}$, then
      $a {\widehat R}^{\bullet} = a F \cap {\widehat R}^{\bullet}$.

The group $G = \mathsf q (F) /\mathsf q ({\widehat R}^{\bullet})$  is isomorphic to $\mathcal{C}_v ({\widehat R}^{\bullet})$, and thus $G$ is finite by Proposition \ref{3.7}.4.
Since ${\widehat R}^{\bullet} \subset F$ is cofinal, it follows that $F/{\widehat R}^{\bullet} = G$  (see the discussion of class groups in Section \ref{2}), and thus the canonical map  $\iota \colon F \to G$  is surjective.

Since $G$ is finite and $F = \bigcup_{g \in G} \iota^{-1}(g) = \bigcup_{g
\in G} g \cap F$, it follows from (\ref{bijection}) that it suffices to prove that, for each $g \in G$, the set
\begin{align*}
\{ y^{-1} \RingReg{R} \cap F \mid y \in g \cap F\}
\end{align*}
is finite.

Now let $g \in G$ be fixed and let $b \in g \cap F$. Set $B = bF \cap
{\widehat R}^{\bullet} \in \mathcal I_v^* ({\widehat R}^{\bullet})$. Then there is a bijective map
      \begin{align*}
        \gamma \colon \begin{cases}
          B^{-1} &\to g \cap F \\
          x &\mapsto bx
        \end{cases}
        \text{.}
      \end{align*}
Indeed, if $x \in B^{-1}$, then $x B \in \mathcal I_v^* ({\widehat R}^{\bullet})$ and thus $xB =
      b' F \cap {\widehat R}^{\bullet}$ for some $b' \in F_0$. If $u \in B$, then $ux \in
      {\widehat R}^{\bullet}$ and therefore
      \[
        uxb F \cap {\widehat R}^{\bullet} =
        ux B
        =
        (uF \cap {\widehat R}^{\bullet}) \cdot_{v_{\widehat{R}^{\bullet}}} (b' F \cap {\widehat R}^{\bullet})
        =
        ub' F \cap {\widehat R}^{\bullet}
        \text{.}
      \]
      Hence $uxb F = ub' F$, which implies $bx \in F$. Since $\iota(bx) =
      \iota(b) \in g$, we obtain $bx \in g \cap F$ and thus $\gamma$ is
      well-defined. Clearly, $\gamma$ is injective.

      For surjectivity, let $b' \in g \cap F$. Then $\iota(b) = \iota(b')$ implies $x = b^{-1} b'
      \in \mathsf q (\RingReg{R})$, and $Bx \subset \mathsf q (\RingReg{R}) \cap bx F \subset
      \mathsf q (\RingReg{R}) \cap F = {\widehat R}^{\bullet}$. Hence $x \in B^{-1}$ and $b' =
      \gamma(x)$. Thus $\gamma$ is surjective.

      Now let $x \in B^{-1}$, then
      \begin{align*}
        \gamma(x)^{-1} \RingReg{R} \cap F
        &=
        b^{-1} (x^{-1} \RingReg{R} \cap bF)
        =
        b^{-1} (x^{-1} \RingReg{R} \cap \mathsf q (\RingReg{R}) \cap F \cap bF) \\
        &=
        b^{-1} (x^{-1} \RingReg{R} \cap {\widehat R}^{\bullet} \cap bF)
        =
        b^{-1} (x^{-1} \RingReg{R} \cap B) \,.
      \end{align*}
      Thus it already suffices to prove that the set $Z = \{ x^{-1} \RingReg{R}
      \cap B \mid x \in B^{-1}\}$ is finite.

      Now set $C = (B^{-1})_{v_{\widehat{R}}}$ and $Z_0
      = \{ x^{-1} R \cap (\widehat{R}\DP C) \mid x \in C^\bullet \}$. Then $C\in\Fv{{\widehat R}}$.
      Since $B^{-1}\in\mathcal{F}_v(\widehat{R}^{\bullet})$, and $\widehat{R}$ is a $v$-Marot ring, it follows by Lemma \ref{3.4}
      that $C^{\bullet}=((B^{-1})_{v_{\widehat{R}}})^{\bullet}=B^{-1}$ and $(\widehat{R}:C)^{\bullet}=(\widehat{R}^{\bullet}:C^{\bullet})$.
      Therefore,
      \begin{align*}
        Z
        &=\{x^{-1}R^{\bullet}\cap (\widehat{R}^{\bullet}:B^{-1})\mid x\in B^{-1}\}=\{x^{-1}R^{\bullet}\cap (\widehat{R}^{\bullet}:C^{\bullet})\mid x\in C^{\bullet}\}\\
        &=\{(x^{-1}R)^{\bullet}\cap (\widehat{R}:C)^{\bullet}\mid x\in C^{\bullet}\}=\{(x^{-1}R\cap (\widehat{R}:C))^{\bullet}\mid x\in C^{\bullet}\}=\{I^{\bullet}\mid I\in Z_0\}.
      \end{align*}
      If $x \in C^\bullet$, then we have the $R$-modules
      $(R \DP C) \subset x^{-1} R \cap (\widehat{R}\DP C) \subset ({\widehat R} \DP C)$. Hence it suffices
      to prove that $({\widehat R} \DP C)/(R \DP C)$ is finite.

      Since $\Fv{{\widehat R}}\subset\Fv{R}$ by Proposition \ref{3.7}.2 and Proposition \ref{3.8} and
      since $R$ is a $v$-Marot Mori ring, there
      exists a finite regular set $E\subset C^{\bullet}$ such that $C=E_{v_R}$. Hence $E_{v_{\widehat R}}
      \subset C$ and since $E_{v_{\widehat R}}\in\Fv{{\widehat R}}\subset\Fv{R}$, it follows that
      $E_{v_{\widehat R}}=(E_{v_{\widehat R}})_{v_R}\supset E_{v_R}=C$ and so $C=E_{v_{\widehat R}}=E_{v_R}$. We now have
      \begin{align*}
        (R \DP C) &= (R \DP E) = \bigcap_{e \in E} e^{-1} R, \text{ and} \\
        ({\widehat R} \DP C) &= ({\widehat R} \DP E) = \bigcap_{e \in E} e^{-1} {\widehat R}
      \end{align*}
      and thus there is a monomorphism $({\widehat R}:C)/(R:C) \to \prod_{e \in E} e^{-1} {\widehat R}
      / e^{-1} R$. Since for every $e \in E$ we have $e^{-1} {\widehat R} / e^{-1} R \cong
      {\widehat R}/R \cong ({\widehat R}/\mathfrak{f}) / (R/\mathfrak{f})$ and this group is finite by
      assumption, so $({\widehat R} \DP C)/(R \DP C)$ is also finite.
\end{proof}

\medskip
\begin{corollary} \label{4.9}
Let $R$ be a $v$-Marot Mori ring such that $(R:\widehat{R})$ is regular and $\widehat{R}/(R:\widehat{R})$ and $\mathcal{C}(\widehat{R})$ are both finite. Let $R\subset A\subset\widehat{R}$ be a  ring such that $A_{v_R}=A$ and let $S\subset R^{\bullet}$ be a multiplicatively closed subset.
\begin{enumerate}
\item $A$ is a $v$-Marot Mori ring, $(A:\widehat{A})$ is regular, and $\widehat{A}/(A:\widehat{A})$ and $\mathcal{C}(\widehat{A})$ are both finite.

\smallskip
\item $S^{-1}R$ is a $v$-Marot Mori ring, $(S^{-1}R:\widehat{S^{-1}R})$ is regular, and $\widehat{S^{-1}R}/(S^{-1}R:\widehat{S^{-1}R})$ and $\mathcal{C}(\widehat{S^{-1}R})$ are both finite.

\smallskip
\item $A$ and $S^{-1}R$ are {\rm C}-rings.
\end{enumerate}
\end{corollary}

\begin{proof}
We set  $T=\mathsf{T}(R)$.

\smallskip
1. Clearly, $\widehat{R}\subset\widehat{A}\subset\widehat{\widehat{R}}=\widehat{R}$, and thus $\widehat{A}=\widehat{R}$. Therefore, $\mathcal{C}(\widehat{A})=\mathcal{C}(\widehat{R})$ is finite. Moreover, $(R:\widehat{R})=(R:\widehat{A})\subset (A:\widehat{A})$, and thus $(A:\widehat{A})$ is regular. We have $\widehat{A}/(A:\widehat{A})=\widehat{R}/(A:\widehat{A})$ is an epimorphic image of $\widehat{R}/(R:\widehat{R})$. This implies that $\widehat{A}/(A:\widehat{A})$ is finite. Observe that $A\in\mathcal{F}_v(R)$. It follows from Proposition \ref{3.8} that $A$ is a $v$-Marot Mori ring.

\smallskip
2. By \cite[Theorem 2.13]{Lu04a} we have $S^{-1}R$ is a Mori ring. It follows from Proposition \ref{3.11} that $S^{-1}R$ is a $v$-Marot ring, $\widehat{S^{-1}R}=S^{-1}\widehat{R}$, and $(R:\widehat{R})\subset S^{-1}(R:\widehat{R})=(S^{-1}R:\widehat{S^{-1}R})$. Therefore, $(S^{-1}R:\widehat{S^{-1}R})$ is regular. Since for every ideal $I$ of a ring $D$ and every multiplicatively closed subset $\mathcal S \subset D$ we have
\[
\mathcal S^{-1}D /  \mathcal S^{-1} I  \cong (\mathcal S + I/I)^{-1} (D/I) \,,
\]
it follows that $\widehat{S^{-1}R}/(S^{-1}R:\widehat{S^{-1}R})=S^{-1}\widehat{R}/S^{-1}(R:\widehat{R})$ is finite. By Proposition \ref{3.8}.4 we have $\widehat{R}$ is a Krull ring. Consequently, Proposition \ref{3.7}.3 implies that $\mathcal{F}_v(\widehat{R})/\{a\widehat{R}\mid a\in T^{\bullet}\}=\mathcal{C}(\widehat{R})$ is finite. It follows from Proposition \ref{3.11}.6 that $\mathcal{C}(\widehat{S^{-1}R})=\mathcal{C}(S^{-1}\widehat{R})\subset\mathcal{F}_v(S^{-1}\widehat{R})/\{aS^{-1}\widehat{R}\mid a\in T^{\bullet}\}$ is finite.

\smallskip
3. This is an immediate consequence of 1., 2. and Theorem \ref{4.8}.
\end{proof}

\bigskip
\noindent
{\bf Acknowledgement.} We thank the referees for their careful reading of the whole manuscript. They made  us aware of some crucial references and their hints allowed us to simplify several of our arguments.

\providecommand{\bysame}{\leavevmode\hbox to3em{\hrulefill}\thinspace}
\providecommand{\MR}{\relax\ifhmode\unskip\space\fi MR }
\providecommand{\MRhref}[2]{%
  \href{http://www.ams.org/mathscinet-getitem?mr=#1}{#2}
}
\providecommand{\href}[2]{#2}

\end{document}